\documentclass[reqno, 12pt]{amsart}
\usepackage[utf8]{inputenc}
\usepackage[english]{babel}
\usepackage{xcolor}
\usepackage[mathscr]{eucal}
  \RequirePackage{amsmath}
\RequirePackage{amssymb} \RequirePackage{amsthm}
\RequirePackage{epsfig}
\usepackage{verbatim}
\newcommand{\hol}{{\mathcal Hol}}

\def\D{{\mathbb D}}
\def\T{{\mathbb T}}

\newcommand{\eiteta}{{e^{i\theta}}} 
\newcommand{\vb}{\varphi_a}

\newtheorem{defn}{Definition}[section]

\newcommand{\la}{\lambda}

\newcommand{\Di}{\ensuremath{{\mathcal D}}} 
\newcommand{\Q}{\ensuremath{{\mathcal Q}}}
\newcommand{\W}{\ensuremath{{\mathcal W}}} 

\newtheorem{lem}{Lemma}

\newtheorem{theorem}{Theorem}

\newtheorem{corollary}{Corollary}
\newtheorem{proposition}{Proposition}
\theoremstyle{definition}

\theoremstyle{remark}

\numberwithin{equation}{section}
\theoremstyle{theorem}
\newtheorem{other}{\bf Theorem}              
\newtheorem{otherl}{\bf Lemma}        

\newcommand{\n}[1]{\|#1\|}

\begin{document}
\title{A family of Dirichlet-Morrey spaces}

\author[P. Galanopoulos]{Petros Galanopoulos}
 \address{Department of Mathematics,
Aristotle University of Thessaloniki,
54124 Thessaloniki,
Greece}
 \email{petrosgala@math.auth.gr}

\author[N. Merch\'an]{Noel Merch\'an}
 \address{Departamento de An\'alisis Matem\'atico, Estad\'istica e Investigaci\'on Operativa,
   y Matem\'atica Aplicada. Facultad
de Ciencias, Universidad de M\'alaga, Campus de Teatinos, 29071
M\'alaga, Spain}
 \email{noel@uma.es}

 \author[A. G. Siskakis]{Aristomenis G.  Siskakis}
 \address{Department of Mathematics,
Aristotle University of Thessaloniki,
54124 Thessaloniki,
Greece}
 \email{siskakis@math.auth.gr}

\date{}

\begin{abstract}
To  each weighted Dirichlet space $\Di_p$, $0<p<1$, we associate a family of Morrey-type spaces $\Di_p^{\la}$, $0< \la < 1$,
constructed by imposing    growth conditions on the norm of hyperbolic translates of  functions. We indicate some of the properties
of these spaces, mention the characterization in terms of boundary values, and study
integration and multiplication operators on them.
\end{abstract}

\thanks{The research of the first and the second author was supported by the grants from Spain MTM2014-52865-P (Ministerio de
Econom\'{\i}a y Competitividad) and FQM-210 (Junta de Andaluc\'{\i}a). The second author was also supported by the grant from Spain FPU2013/01478 (Ministerio de
Educaci\'{o}n, Cultura y Deporte).}

\maketitle

\bigskip
\section{Introduction}\label{intro}
Let $\D$ be the unit disc in the complex plane and $\hol(\D)$ be the space of analytic functions in $\D$.
 A function $f\in\hol(\D)$ belongs to the Hardy space $H^2$ if
 $$
 \|f\|_{H^2}^2=
 \sup_{r\in (0,1)}\frac{1}{2\pi} \int_0^{2\pi} |f(r\eiteta)|^2\,d\theta<\infty\,.
 $$
If $f\in H^2$ then the  radial limits
$$
 \lim_{r\to 1} f(re^{i\theta})=f(e^{i\theta})
 $$
  exist almost everywhere on the unit circle $\T$  and
 \begin{equation}\label{Hardy Norm}
\|f\|_{H^2}^2=\frac{1}{2\pi} \int_0^{2\pi} |f(\eiteta)|^2\,d\theta.
\end{equation}

Recall that the space  $BMOA$ consists of all functions $f\in H^2$ whose boundary values  $f(e^{i\theta})$ have  bounded
mean oscillation, that is,
\begin{equation}\label{BMO}
\sup_{I\subseteq\T} \frac{1}{|I|} \int_I \left|f(e^{i\theta})- f_I\right|^2\, d\theta <\infty,
\end{equation}
where $f_I=\frac{1}{|I|}\int_I f(e^{it})\,dt$ is the average of $f$ over $I$, and
the supremum is taken over all subarcs $I$ of $\T$ with length $|I|$.
There are other equivalent  definitions for $BMOA$ one of which is the following.
For $a\in \D$ let
$$
\varphi_a(z)=\frac{a-z}{1-\bar{a}z}, \quad z\in\D,
$$
be the analytic automorphism of $\D$ which exchanges $0$ with  $a$,  and for $f\in H^2$ consider the set of
hyperbolic translates of $f$,
$$
S(f)=\{f\circ\varphi_a -f(a):  a\in \D\}.
$$
$BMOA$ then can be defined as the space of   all $f\in H^2$ such that
\begin{equation}\label{BMOA NORM}
\n{f}_{BMOA}=|f(0)|+  \sup\limits_{a\in\D} \|f\circ\vb- f(a)\|_{H^2}<\infty,
\end{equation}
i.e. $f\in BMOA$ if and only if the set $S(f)$ is  bounded  in the norm of $H^2$.
The above quantity defines a norm making $BMOA$ a Banach space.
More information on $BMOA$  can be found in  \cite{Gi}.

\textbf{Morrey spaces}. Morrey spaces were introduced in the 1930's in connection to
partial differential equations, and were subsequently studied as function classes in harmonic analysis on Euclidean spaces.
The analytic Morrey spaces were introduced recently and studied  by several authors,
see for example \cite{WX2}, \cite{XX}, \cite{XY} and \cite{Liu2}.

We  recall the definition  and some of their properties.
Observe  that
$$
\int_I \left|f(e^{i\theta})- f_I\right|^2\, d\theta \to 0 \qquad \text{as} \,\,|I|\to 0,
$$
for every $f\in H^2$, and the rate of this convergence to $0$ depends clearly on the  degree of oscillation of
$f$ around its average $f_I$ on $I$.
Given a $\la\in (0,1)$ we can isolate  functions $f$ for which this rate of convergence is comparable to   $|I|^{\la}$.
Thus for $f\in H^2$ we set
\begin{equation}\label{HardyMorrey}
\n{f}_{\la, \ast}=\sup_{I\subseteq\T} \left(\frac{1}{|I|^{\la}} \int_I \left|f(e^{i\theta})- f_I\right|^2\, d\theta\right)^{1/2},
\end{equation}
and define the space
\begin{equation}\notag
    H^{2, \la}=  \{f\in H^2:  \n{f}_{ \la, \ast}<\infty\}.
\end{equation}
This is a linear space. The seminorm $\n{\,\,}_{ \la, \ast}$ can be completed to a norm by adding  $|f(0)|$ to it,
making $H^{2, \la}$ into  a Banach space.

It is clear that for $\la=0$ or $\la=1$, $H^{2,\la}$
 reduces to  $H^2$ and $BMOA$ respectively, and for $0<\la<1$,
$$
BMOA\subset H^{2, \la}\subset H^2.
$$
Furthermore the following Carleson measure characterization holds
\begin{equation}\label{Carl-Morrey}
f\in H^{2, \la} \Leftrightarrow
\sup_{I\subset\T}\frac{1}{|I|^{\la}}\int_{S(I)}|f'(z)|^2(1-|z|^2)\,dm(z)<\infty,
\end{equation}
where
$$
S(I)=\{re^{i\theta}: 1-|I|\leq r<1,\, e^{i\theta}\in I\}
$$
are  the Carleson boxes based on arcs $I\subset \T$ and $dm(z)$ is the
 planar Lebesgue measure normalized so that the area of the unit disc is $1$.
This and several other properties of Morrey spaces can be found in \cite{XX}, \cite{XY} or \cite{XiG}. A characterization of $H^{2, \la}$ analogous to (\ref{BMOA NORM})  says that  if $f\in H^2$ then
 $f\in H^{2, \la}$ if and only if
\begin{equation}\label{Morrey-Mobius}
\n{f}_{{\la, \ast\ast}}=
\sup_{a\in \D}(1-|a|^2)^{\frac{1}{2}(1-\la)}\n{f\circ\varphi_a - f(a)}_{H^2}<\infty,
\end{equation}
and   $\n{f}_{H^{2, \la}}=|f(0)|+\n{f}_{{\la, \ast\ast}}$ is a norm, equivalent to  $|f(0)|+\n{f}_{ \la, \ast}$. Equivalently $H^{2, \la}$ can be constructed as the subspace of  $H^2$ containing the functions whose conformal translates have $H^2$ norms of restricted growth,
\begin{equation}\label{M-growth}
\n{f\circ\varphi_a - f(a)}_{H^2}\leq \frac{C}{(1-|a|^2)^{\frac{1}{2}(1-\la)}}, \,\, |a|\to 1,
\end{equation}
with the constant $C$ depending only on $f$.

\textbf{General Morrey-type spaces}. We take the opportunity to notice that  the above construction can be carried out in more general terms.
Suppose $X$ is a Banach space of analytic functions on $\D$ which contains the constant functions and such that   point evaluations $f\to f(a)$, $a\in\D$, are continuous linear functionals on $X$.
 Suppose also that $w:\D\to (0, \infty)$ is an appropriate weight function. Consider the  Morrey-type space generated by  $(X, w)$ which is  defined to be the space $M(X, w)$ of functions $f\in X$ such that $f\circ\varphi_a\in X$ for $a\in \D$ and
for which there is  a constant $C=C(f)$ such that
$$
 \n{f\circ\varphi_a -f(a)}_X\leq Cw(a), \quad a\in \D.
$$
Without any restrictions on  $w$ the  space $M(X, w)$ may reduce to $M(X, w)=\{0\}$ or $M(X, w)=X$ or it may consist  only of constant functions.  But generally there are  weights, appropriate for the base space $X$,  for which $M(X, w)$ is a nontrivial  proper subspace of $X$. For example if $w\equiv 1$ on $\D$ then this construction gives the M\"{o}bius invariant spaces $M(X)$ generated by $X$  considered in \cite{AS}, a particular case of which is $M(H^2)=BMOA$. A convenient class of  weights that can be considered in this construction are the  radial weights, $w(a)=w(|a|)$,  and it may  be assumed further that   $w(|a|)$  is  nondecreasing in $|a|$. A particular such   family is
\begin{equation}\label{w-power}
w(a) = (1-|a|)^{-s}, \quad a\in \D,
\end{equation}
with $s\geq 0$. The specific choice   $w_{\la}(a) =(1-|a|)^{- \frac{1}{2}(1-\la)}$ gives
$$
M(H^2, w_{\la})= H^{2, \la}, \quad 0\leq \la \leq 1.
$$

In the general case, if $(X,w)$ is a pair for which the resulting space $M(X,w)$ is nontrivial, then the quantity
$$
\n{f}_{M(X, w)}=|f(0)|+\sup_{a\in\D}\frac{1}{w(a)}\n{f\circ\varphi_a - f(a)}_X
$$
is a norm on  $M(X,w)$ and makes it into a Banach space. Interesting questions arise such as characterizing  $M(X,w)$ by  Carleson measure type conditions or in terms of boundary values of its functions. We will not pursue this further here, but will  concentrate instead on a  family of spaces obtained when $X$ is a weighted Dirichlet space and $w$ is of the form (\ref{w-power}).

\textbf{Dirichlet spaces and $\Q_p$ spaces}.
 Recall the following estimate for the norm of a function $f\in H^2$,
$$
\n{f}_{H^2}^2 \sim |f(0)|^2+\int_\D |f'(z)|^2(1-|z|^2)\,dm(z),
$$
where $\sim $ means that each of the two quantities is dominated by a
constant multiple of the other for all $f\in H^2$.
If the weight $1-|z|^2$  inside the integral is
 replaced by $\log(1/|z|^2)$ then the above estimate becomes an  identity valid for all $f\in H^2$,  known as the
 Littlewood-Paley identity.

 The weighted Dirichlet spaces $\Di_p$, $0\leq p<\infty$, are defined
 to contain those $f\in \hol(\D)$ for which
$$
\n{f}_{\Di_p}^2=|f(0)|^2+\int_\D |f^\prime(z)|^2(1-|z|^2)^p\, dm(z)<\infty.
$$
This quantity is a norm. Clearly $\Di_1=H^2$ with equivalence of norms, and $\Di_0$ is the classical
Dirichlet space   denoted by $\Di$. For $p>1$, $\Di_p$ coincides with
a  weighted Bergman space with weight $(1-|z|)^{p-2}$. If  $0<p<q $ then
$$
\Di \subset \Di_p\subset \Di_q,
$$
and there is a constant $C=C(p,q)$ such that $\n{f}_{\Di_q}\leq C\n{f}_{\Di_p}$ for each $f\in\Di_p$.

As in the case of $H^2$ we can consider the M\"{o}bius invariant version of $\Di_p$,
that is the subspace of $\Di_p$ which consists of
all $f\in \Di_p$ such that the set  $S(f)=\{f\circ\varphi_a -f(a):  a\in \D \}$ is bounded in $\Di_p$. These are the
spaces $\Q_p$,  originally defined and studied by Aulaskari, Xiao and Zhao \cite{AXZ}. Under  the norm
\begin{equation}
\n{f}_{\Q_p} =|f(0)|+\sup_{a\in\D} \|f\circ\vb- f(a)\|_{\Di_p},
\end{equation}
they are  Banach spaces and we have $\Q_0=\Di$, $\Q_1=BMOA$,  while for  all $p>1$,  $\Q_p$ coincides with the Bloch space
$\mathcal{B}$ of functions that satisfy $ \sup_{z\in\D}(1-|z|^2)|f'(z)|<\infty$, see \cite{Xi}. For the remaining values  $0<p<1$ the
resulting spaces satisfy
$$
\Di\subset \Q_p\subset BMOA
$$
and they form a strictly increasing  chain of M\"{o}bius invariant spaces, characterized by the Carleson measure condition,
$$
f\in \Q_p \Leftrightarrow \sup_{I\subset\T}\frac{1}{|I|^p}\int_{S(I)}|f'(z)|^2(1-|z|^2)^p\,dm(z)<\infty.
$$
For information on these spaces see \cite{Xi}  and \cite{XiG} and the references therein.

\textbf{Dirichlet-Morrey spaces}. Let $0\leq p\leq 1$ and  $f\in\Di_p$. The
following estimate is valid
$$
\n{f\circ\varphi_a-f(a)}_{\Di_p} \leq \frac{C\n{f}_{\Di_p}}{(1-|a|^2)^{\frac{p}{2}}}, \quad a\in \D,
$$
with the constant $C$ depending only on $p$.  In analogy with the construction of
$H^{2, \la}$ from the Hardy  space $H^2=\Di_1$ we define  the Dirichlet-Morrey spaces as follows.

\begin{defn}
Let $\la,\,p\in [0, 1]$. We say that an $f\in\hol(\D)$ belongs to the Dirichlet-Morrey space $\Di_p^{\la}$ if
\begin{equation}\label{DIRMORNORM}
\|f\|_{\Di_p^{\la}}= |f(0)|+\sup_{a\in\D} (1-|a|^2)^{\frac{p}{2}(1-\la)}\n{f\circ \varphi_a-f(a)}_{\Di_p}<\infty.
\end{equation}
\end{defn}

It is clear $\Di_p^{\la}$ is a linear space and the above quantity is a norm, under which $\Di_p^{\la}$ is a Banach space.
We see that  $\Di_1^{\la}=H^{2, \la}$ and that for each  $p$, $\Di_p^{1}=\Q_p$ and    $\Di_p^{0}=\Di_p$ and we have
$$
\Q_p\subseteq \Di_p^{\la} \subseteq \Di_p,\qquad 0<\la <1.
$$

In the rest of the article we will state some basic properties of Dirichlet-Morrey spaces, discuss briefly their characterization in terms of
boundary values and concentrate in Section 3 on the boundedness of integration operators and pointwise multipliers.

We will write $A\lesssim B $ between two quantities if there is a constant $C$ such that $A\leq C B$ for all values
of the parameter involved in the quantities $A, B$.
If both $A\lesssim  B$ and $B\lesssim  A$ are valid we write $A\sim B$. When a constant $C$ appears, its value
may be different from one step to the next.

\section{Some basic properties}

The following proposition gives a Carleson measure characterization of $\Di_p^\la$, which is analogous to
(\ref{Carl-Morrey}) for Hardy-Morrey spaces

 \begin{proposition}\label{Carl-DM}
Let $0<p, \la< 1$ and $f\in\hol(\D)$. Then the following are equivalent,
\begin{enumerate}
\item  $f\in \Di_p^\la$.
\item   $\n{f}_{p,\la, \ast}=\sup\limits_{I\subset\T} \left(\frac{1}{|I|^{p\la}}\int_{S(I)} |f'(z)|^2(1-|z|^2)^p\,dm(z)\right)<\infty,$
\end{enumerate}
and the norm $\n{f}_{\Di_p^{\la}}$ is comparable to $|f(0)|+\n{f}_{p,\la,\ast}$.
\end{proposition}

\begin{proof} Assume $f\in\Di_p^\la$. For an interval  $I\subset\T$ let $\zeta$ be the midpoint of $I$ and let $a=a_I=(1-|I|)\zeta$.
Note that
$$
|1-\overline{a}z|\sim |I|=1-|a|\sim 1-|a|^2, \qquad z\in S(I),
$$
thus
\begin{align*}
&\frac{1}{|I|^{p\la}}\int_{S(I)} |f'(z)|^2(1-|z|^2)^p\,dm(z) \\
&\sim
\frac{1}{|I|^{p\la}} \int_{S(I)} |f'(z)|^2 (1-|z|^2)^p\frac{(1-|a|^2)^{2p}}{|1-\overline{a}z|^{2p}}\,dm(z)\\
&\sim |I|^{p(1-\la)} \int_{S(I)} |f'(z)|^2(1-|\varphi_{a}(z)|^2)^p\,dm(z)\\
&\lesssim (1-|a|^2)^{p(1-\la)}\int_{\D} |f'(z)|^2(1-|\varphi_{a}(z)|^2)^p\,dm(z)\\
&=(1-|a|^2)^{p(1-\la)}\int_{\D} |(f\circ\varphi_a)'(z)|^2(1-|z|^2)^p\,dm(z)\\
& =(1-|a|^2)^{p(1-\la)}\n{f\circ\varphi_a-f(a)}_{\Di_p}^2\\
& \leq \n{f}_{\Di_p^{\la}}^2.
\end{align*}
This is valid for each interval  $I\subset\T$ and taking supremum shows that (1) implies (2).

Conversely suppose (2)  holds. That is,   for the nonnegative measure $d\mu(z)=|f'(z)|^2(1-|z|^2)^p\,dm(z)$
there is a constant $C$ such that
$$
\mu(S(I))=\int_{S(I)} d\mu(z)\leq C|I|^{p\la}
$$
 for all $I\subset \T$, i.e. $d\mu$ is a $p\la$-Carleson measure.  Then for $a\in \D$,
\begin{align*}
 \n{f\circ\varphi_a-f(a)}_{\Di_p}^2&=\int_{\D}|f'(z)|^2(1-|\varphi_a(z)|^2)^p\,dm(z)\\
 &=\int_{\D} |f'(z)|^2\frac{(1-|z|^2)^p(1-|a|^2)^p}{|1-\overline{a}z|^{2p}}\,dm(z)\\
 &= \int_{\D}\frac{(1-|a|^2)^p}{|1-\overline{a}z|^{2p}}\,d\mu(z).
\end{align*}
Thus we have
\begin{align*}
\sup_{a\in \D}(1-|a|^2)^{p(1-\la)}\n{f\circ\varphi_a-f(a)}_{\Di_p}^2
&=\sup_{a\in \D} \int_{\D}\frac{(1-|a|^2)^{2p-p\la}}{|1-\overline{a}z|^{2p}}\,d\mu(z)\\
&= \sup_{a\in \D}\int_{\D}\frac{(1-|a|^2)^q }{|1-\overline{a}z|^{q+p\la}}\,d\mu(z)\\
&<\infty,
\end{align*}
by using the characterization of Carleson  measures in \cite[Lemma 3.1.1] {XiG} with  $q=2p-p\la>0 $, completing the proof.
\end{proof}

 \begin{proposition}
Let $0< p, \la< 1$ then,
\begin{enumerate}
\item There is a constant $C=C(p, \la)$ such that any   $f\in\Di_p^{\la}$ satisfies
\begin{equation}\label{growth}
|f(z)|\le \frac{C\n{f}_{\Di_p^{\la}}}{(1-|z|)^{\frac{p}{2}(1-\la)}}, \quad z\in \D
\end{equation}
\item The function $f_{p,\la}(z)= (1-z)^{-\frac{p}{2}(1-\la)}$ belongs to $\Di_p^{\la}$.
\end{enumerate}
\end{proposition}

\begin{proof} (i) Suppose $f\in\Di_p^{\la}$.
We apply the inequality
$$
|g(0)|^2\leq (p+1) \int_{\D}|g(z)|^2(1-|z|^2)^p\,dm(z),
$$
see \cite[Lemma 4.12]{Zhu},  valid for all  analytic $g$ on $\D$, to the function
$g=(f\circ\varphi_w-f(w))'$ to obtain
\begin{align*}
|f'(w)|^2 (1-|w|^2)^2 &\leq (p+1) \int_{\D}|(f\circ\varphi_w)'(z)|^2(1-|z|^2)^p\,dm(z)\\
&= \frac{p+1}{(1-|w|^2)^{p(1-\la)}}  \left((1-|w|^2)^{p(1-\la)}\n{f\circ\varphi_w-f(w)}_{\Di_p}^2\right)\\
&\leq\frac{p+1}{(1-|w|^2)^{p(1-\la)}} \n{f}_{\Di_p^\la}^2,
\end{align*}
for each $w\in\D$. Thus
$$
|f'(w)|\leq\frac{(p+1)^{1/2}}{(1-|w|^2)^{1+\frac{p}{2}(1-\la)}} \n{f}_{\Di_p^\la},\quad w\in \D.
$$
Using this and the integration  $f(z)-f(0)=\int_0^zf'(\zeta)\,d\zeta$ we obtain  the desired growth inequality.

(ii) We will verify that $|f'_{p, \la}(z)|^2(1-|z|^2)^p\,dm(z)$ is a $p\la$-Carleson measure and then
Proposition \ref{Carl-DM} gives the conclusion.
In doing so, it is more convenient to work with the equivalent family of Carleson lune-shaped sets $S(b, h)=\{z\in \D: |b-z|<h\}$,
where $b\in\T$ and $0<h<1$,  than with the Carleson boxes $S(I), I\subset\T$. Thus it suffices to show that
\begin{equation}\label{eq-lunes}
\sup_{b\in\T, 0<h<1}\frac{1}{h^{p\la}}\int_{S(b, h)}|f'_{p, \la}(z)|^2 (1-|z|^2)^p\,dm(z)<\infty.
\end{equation}
We have
\begin{align*}
\int_{S(b, h)}|f'_{p, \la}(z)|^2(1-|z|^2)^p\,dm(z)&=
 C_1\int_{S(b, h)}\frac{(1-|z|^2)^p}{|1-z|^{2+p(1-\la)}}\,dm(z)\\
 &\lesssim \int_{S(b, h)}\frac{1}{|1-z|^{2-p\la}}\,dm(z)\\
 &\lesssim  \int_{S(1, h)}\frac{1}{|1-z|^{2-p\la}}\,dm(z)\\
 &\lesssim \int_{|w|<h}\frac{1}{|w|^{2-p\la}}\,dm(w)\\
 &= \int_0^h\frac{1}{r^{1-p\la}}dr\\
 &= h^{p\la}.
\end{align*}
Thus (\ref{eq-lunes}) holds and the proof is finished.
\end{proof}
Observe that both parts of the above Proposition are also valid when $p=1, 0<\la<1$.

\begin{proposition}
Let $\la_1, p_1, \la_2, p_2 \in(0,1)$. Then
$$
\Di_{p_1}^{\la_1} \subseteq \Di_{p_2}^{\la_2}
\Longleftrightarrow p_1\le p_2 \text{ and } p_1(1-\la_1)\le p_2(1-\la_2).
$$
\end{proposition}

\begin{proof}
Assume  $ p_1\le p_2$ and $ p_1(1-\la_1)\le p_2(1-\la_2)$ and let $f\in\Di_{p_1}^{\la_1}$ and  $I\subset\T$. Then
\begin{align*}
\frac{1}{|I|^{ p_2\la_2}}&\int_{S(I)} |f'(z)|^2(1-|z|^2)^{ p_2}\, dm(z)\\
&=\frac{1}{|I|^{ p_2\la_2}}\int_{S(I)} |f'(z)|^2(1-|z|^2)^{ p_1}(1-|z|^2)^{ p_2- p_1}\, dm(z)\\
&\leq \frac{|I|^{ p_2- p_1}}{|I|^{ p_2\la_2}}\int_{S(I)} |f'(z)|^2(1-|z|^2)^{ p_1}\, dm(z)\\
&= |I|^{ p_2(1-\la_2)- p_1(1-\la_1)}\left(\frac{1}{|I|^{ p_1\la_1}}\int_{S(I)} |f'(z)|^2(1-|z|^2)^{ p_1}\, dm(z) \right)
\end{align*}
and by Proposition \ref{Carl-DM}, it follows  $\Di_{p_1}^{\la_1} \subseteq \Di_{p_2}^{\la_2}$.

 Assume now that $\Di_{p_1}^{\la_1} \subseteq \Di_{p_2}^{\la_2}$.
Then it is necessary that $ p_1\le p_2$. The easiest way to see this is to
use the class $HG$ of  functions in $\hol(\D)$ whose Taylor series with center at $0$ has Hadamard gaps. According
to \cite[Theorem 1.2.1]{Xi} for $0<p<1$ we have $HG\cap \Q_p= HG\cap \Di_p$, and for  $0<p<q<1$ we have
$HG\cap \Di_p\subsetneq HG\cap \Di_q$. If we assume
 that $p_2<p_1$ then   $\Di_{p_2}\subseteq \Di_{p_1}$ and using the assumption
 $\Di_{p_1}^{\la_1} \subseteq \Di_{p_2}^{\la_2}$  we will have further
$\Q_{ p_1}\subseteq\Di_{p_1}^{\la_1}\subseteq \Di_{p_2}^{\la_2}
\subseteq \Di_{p_2} \subseteq \Di_{ p_1}$. This would imply  that
$HG\cap \Di_{p_1}=HG\cap \Di_{p_2}$
which contradicts  part of the above mentioned theorem. In addition from (\ref{growth}) it follows easily
that $p_1(1-\la_1)\leq  p_2(1-\la_2)$.
\end{proof}

We next discuss the boundary values characterization of Dirichlet-Morrey spaces. Recall the corresponding result
for Dirichlet spaces and $\Q_p$ spaces from \cite[Lemma 6.1.1.]{Xi}.
If  $f\in H^2$ and   $0<p<1$, then  $f\in \Di_p$ if and only if
$$
\int_{\T}\int_{\T} \frac{|f(u)-f(v)|^2}{|u-v|^{2-p}}\,|du||dv|<\infty,
$$
where the simplified notation is  $u=e^{i\theta}\in\T$ and $|du| =d\theta$. This result together
with the fact that   $\Q_p$ is the M\"{o}bius invariant
version of $\Di_p$, is used to prove Theorem 6.1.1. in \cite{Xi} which says that  for $0<p<1$, $f\in \Q_p$ if and only if
$$
\sup_{I\subset\T} \frac{1}{|I|^p}\int_{I}\int_{I}\frac{|f(u)-f(v)|^2}{|u-v|^{2-p}}\,|du||dv|<\infty,
$$
where the supremum is taken over all arcs $I\subset\T$. The proof of this result is rather long and technical. But it is
easily adapted without any new conceptual or technical requirements  to obtain the following characterization
of Dirichlet-Morrey spaces.

\begin{theorem}
Suppose $f\in H^2$ and let $0<p, \la <1$. Then $f\in \Di_p^{\la}$ if and only if
$$
\sup_{I\subset\T}\frac{1}{|I|^{ p\la}}\int_I\int_I\frac{|f(u)-f(v)|^2}{|u-v|^{2- p}}\,|du|\,|dv|<\infty.
$$
\end{theorem}

\section{Pointwise multipliers}

Let $X$ be a Banach space of analytic functions on $\D$. A function $g\in \hol(\D)$
is said to be a multiplier of $X$ if the multiplication operator
$$
M_g(f)(z) =g(z)f(z), \quad f\in X
$$
is a bounded operator on  $X$.
For this it is  usually enough to check that  $M_g(X)\subset X$ and apply the closed graph theorem. The space of all
multipliers of $X$ is  denoted by $M(X)$. Multiplication operators are closely related
to  integration operators $J_g$ and $I_g$. These are induced by symbols  $g\in \hol(\D)$ as follows
$$
J_g(f)(z)=\int_0^z f(w)g'(w)\,dw, \quad z\in \D,
$$
and
$$
I_g(f)(z)=\int_0^z f'(w)g(w)\,dw, \quad z\in \D,
$$
 and act on functions $ f\in\hol(\D) $.  The operators $I_g, J_g$  have been studied in a number of papers, see
for example \cite{ALC}, \cite{ALS}, \cite{GaGiPe},  \cite{Gi} and \cite{Li}. Their relation to $M_g$ comes from the integration
by parts formula
\begin{equation}\label{MulVol}
  J_g(f)(z) =M_g(f)(z)-f(0)g(0)-I_g(f)(z).
\end{equation}
 This essentially says that  if $g$ is a symbol for which two of the operators $I_g, J_g, M_g$ are bounded on a space  $X$
  so is the third. It also says that it is possible for two of the operators to be  unbounded but the third is bounded
 due to  cancellation.

The space of  multipliers is known for several of the classical spaces such as Hardy and Bergman spaces.
In particular for $H^2=\Di_1$ the space of multipliers is
 $M(H^2)=H^{\infty}$, the algebra of bounded analytic functions. For other Dirichlet spaces $\Di_p$,  $p \in (0,1)$,
the situation is more complicated. The   description of $M(\Di_p)$ is in terms of  $\Di_p$-Carleson measures.
Recall that a positive Borel measure $\mu$ on the disc is a $\Di_p$-Carleson measure if there is a constant  $C=C(\mu)$
such that
$$
\int_\D |f(z)|^2\,d\mu(z)\leq C \|f\|^2_{\Di_p},\qquad  f\in\Di_p .
$$

These measures  were  described  initially by Stegena \cite{St} with the help of Bessel capacities, and similar
characterizations were given by  other authors.

In another approach,  Arcozzi, Rochberg and Sawyer \cite{ARS} described these measures by a different condition, a simplified
form of which is given in  \cite{GP}. Accordingly a finite $\mu$ is $\Di_p$-Carleson if and only if
\begin{equation}\label{GPCM}
\sup_{w\in\D} \frac{1}{\mu(S(w))} \int_{S(w)} \frac{(\mu(S(z)\cap S(w)))^2}{(1-|z|^2)^{2+ p}}\, dm(z) <\infty,
\end{equation}
where for $w\in \D$ the set $S(w)$ on which  integration takes place is the
Carleson box $S(w)=\{z\in\D : 1-|z|\leq 1-|w|,\, |arg(\bar{z}w)|\leq \pi(1-|w|)\}$.

It is convenient at this point to use the space $\W_p$ of functions
$g\in \hol(\D)$ such  that the measure
$$
d\mu_g(z) = |g'(z)|^2(1-|z|^2)^p\,dm(z)
$$
is a $\Di_p$-Carleson measure. This space has been studied \cite{RW} and \cite{W}.
The multipliers of $\Di_p$ were described in  \cite{St} as follows.

 \begin{other}\label{STG}
Suppose $0<p<1$ and   $g\in\hol(\D)$. Then  $g \in M(\Di_p)$
 if and only if $g\in H^{\infty}$  and  $d\mu_g(z) = |g'(z)|^2 (1-|z|^2)^{ p}\,dm(z)$ is a
$\Di_p$-Carleson measure. In other words,
$$
M(\Di_p)= H^{\infty}\cap \W_p.
$$
\end{other}

 On the other hand  the multipliers of $\Q_p$ are completely described in \cite{PP}, \cite{Xi2} as follows.
 \begin{other}\label{PPX} Suppose $0<p<1$ and  $g\in\hol(\D)$. Then $g\in M(\Q_p)$
 if and only if $g\in H^{\infty}$  and
 \begin{equation}\label{LOGCM}
 \sup_{I\subseteq\mathbb T} \frac{\left(\log\frac{1}{|I|}\right)^2}{|I|^{ p}}\int_{S(I)}|g'(z)|^2 (1-|z|^2)^{ p}\,dm(z) <\infty .
 \end{equation}
 \end{other}
Thus if we denote by $\Q_{p, \log}$ the space of functions that satisfy (\ref{LOGCM}) then the above theorem says
$$
M(\Q_p)= H^{\infty}\cap \Q_{p, \log}.
$$
It is not difficult to check that $\Q_{p, \log}\subset \W_p$. On the other hand it was shown in
\cite{ARS} that $\W_p\subset \Q_p$  and there is a simplified proof of this  in  \cite[Lemma 4]{LLZ}. Thus we have
$$
\Q_{p, \log}\subset \W_p\subset \Q_p, \quad 0<p<1.
$$

 In what follows we study the action of the operators $I_g, J_g$ on the spaces ${\Di_p^{\la}} $, and obtain
 information on pointwise multipliers. We will need  the following technical lemma from \cite{PZ} (p. 488). We state only
 the part of it that we need.
\begin{otherl} \label{PZh}
Let $u\in \D$, $|v|\le 1$ and $s>-1,\, r,\,t >0$. Then
\begin{align*}
\int_{\D} &\frac{(1-|z|^2)^s}{|1-\bar{u}z|^r|1-\bar{v}z|^t}\,dm(z)\leq \frac{C}{(1-|u|^2)^{r+t-s-2}}\,, \quad 0<r+t-s-2 <r\,
\end{align*}
where $C$ is an absolute, positive constant.
\end{otherl}

Using this estimate we obtain a family of test functions in  ${\Di_p^{\la}}$.

\begin{lem}\label{LemFunc}
Let $0< p,\la<1$ and $c\in\D$. Then the functions
\begin{equation}\label{f2}
 f_c(z)=\frac{1}{(1-\overline{c}z)^{ p(1-\la)/2}}
\end{equation}
belong to $ \Di_p^{\la}$ and $K=\sup_{c \in \D}\|f_c\|_{{\Di_p^{\la}}}< \infty .$
 \end{lem}
\begin{proof}
Fix  $c\in\D$. Then for $a\in\D$,

\begin{align*}
&(1-|a|^2)^{p(1-\la)}\int_\D |f'_c(z)|^2(1-|\vb(z)|^2)^ p\,dm(z)\\
& = (1-|a|^2)^{ p(2-\la)}
\int_\D \frac{(1-|z|^2)^ p}{|1-\overline{c}z|^{2+ p(1-\la)}|1-\overline{a}z|^{2 p}}\,dm(z)\,.
\end{align*}
Now for $r=2 p,\, t=2+ p(1-\la),\, s= p$,
Lemma \ref{PZh} gives the desired result.
\end{proof}

\begin{theorem}\label{thmIg}
Let $0< p,\la<1$ and $g\in\hol(\D)$. Then $I_g:{\Di_p^{\la}}\to{\Di_p^{\la}}$
 is bounded if and only if  $g\in H^\infty$.
\end{theorem}

\begin{proof}
Let $g\in H^\infty$ then
\begin{align*}
\|I_g(f)\|_{\Di_p^{\la}}^2 &\sim \sup\limits_{I\subset\T}
 \frac{1}{|I|^{ p\la}}\int_{S(I)} |I_g(f)^\prime(z)|^2 (1-|z|^2)^ p\,dm(z)\\
& = \sup\limits_{I\subset\T}\frac{1}{|I|^{ p\la}}\int_{S(I)} |f^\prime(z)|^2|g(z)|^2(1-|z|^2)^ p\,dm(z)\\
 &\lesssim \|g\|_\infty^2\|f\|_{\Di_p^{\la}}^2
\end{align*}
for every $ f\in{\Di_p^{\la}}$. So $\|I_g\|\le C\|g\|_\infty$ where $C$ is a constant.

\par On the other hand, assume that $I_g$ is bounded on ${\Di_p^{\la}}$. We will use the test
functions $f_c$ of Lemma \ref{LemFunc} for $\{|c|> \frac 12\}$. Then
from the Lemma there is a constant $C$ such that $1\leq \n{f_c}_{\Di_p^{\la}}\leq C$
for all $c$, so that $ \|I_g\|^2 \geq \frac{1}{C^2}\n{I_g(f_c)}_{\Di_p^{\la}}^2$ and,
\begin{align*}
\n{I_g(f_c)}_{\Di_p^{\la}}^2&=
\sup\limits_{a\in\D}\,(1-|a|^2)^{ p(1-\la)}\int_\D |I_g(f_c)^\prime(z)|^2(1-|\varphi_a(z)|^2)^ p\, dm(z)\\
& \gtrsim   (1-|c|^2)^{ p(1-\la)}\int_\D |I_g(f_c)^\prime(z)|^2(1-|\varphi_c(z)|^2)^ p\, dm(z)\\
 &= (1-|c|^2)^{ p(1-\la)}\int_\D |f_c^\prime(z)|^2|g(z)|^2(1-|\varphi_c(z)|^2)^ p\, dm(z)\\
&\sim
|c| (1-|c|^2)^{ p(1-\la)}\int_\D \frac{|g(z)|^2(1-|\varphi_c(z)|^2)^ p}{\left|1-\bar{c}z\right|^{2+p(1-\la)}}\, dm(z)\,,
\end{align*}
now by restricting the above integral on a disc with center the point $c$ and radius $\frac{1-|c|}{2}$
and by applying  the mean value property of subharmonic functions we get
that
\begin{align*}
\|I_g\|^2 & \gtrsim |g(c)|^2
\end{align*}
for any $\{|c|> \frac 12\}$. It follows  that $g$ is a bounded analytic function on $\D$.
\end{proof}

\par Concerning  the action of  $J_g$ on ${\Di_p^{\la}}$ we have the following necessary condition.
\begin{theorem}\label{NJg}
Let $0< p,\la<1$ and $g\in\hol(\D)$. If $ J_g : {\Di_p^{\la}} \to {\Di_p^{\la}} $
is bounded then $g \in \Q_p$.
\end{theorem}
\begin{proof}
We use the test functions $f_c(z)=(1-\overline{c}z)^{-p(1-\la)/2}$ of Lemma \ref{LemFunc}.
From the hypothesis there is a constant $C$ such that
$$
\n{J_g(f_c)}_{\Di_p^{\la}}\leq C\n{f_c}_{\Di_p^{\la}}\leq C\sup_{c\in\D}\n{f_c}_{\Di_p^{\la}}=CK<\infty,
$$
for all $c\in\D$. This means that
$$
\sup\limits_{I\subset\T} \frac{1}{|I|^{ p\la}}\int_{S(I)} |f_c(z)|^2|g^\prime(z)|^2(1-|z|^2)^ p\,dm(z)\leq K'<\infty
$$
for all $c\in\D$. For each interval  $I$   choose $c=c_I=(1-|I|)\eiteta$ where $\eiteta$ is the center of $I$, then
$ |1-\bar{c}z|\sim |I|$ for $ z\in S(I)$ and we have
\begin{align*}
K'\geq &\frac{1}{|I|^{ p\la}}\int_{S(I)} \frac{1}{ |1-\overline{c}z|^{p(1-\la)}}|g^\prime(z)|^2(1-|z|^2)^ p\,dm(z) \\
\sim&\frac{1}{|I|^{ p}}\int_{S(I)} |g'(z)|^2(1-|z|^2)^ p\,dm(z)
\end{align*}
with $K'$ independent of $I$. Taking the supremum of the last integral over all $I\subset\T$ we see that $g\in \Q_p$.
\end{proof}

\par We now find  sufficient conditions on $g$ for  $J_g$ to be bounded on $\Di_p^{\la}$.

\begin{theorem}\label{ThmOp} Suppose $0<p<1$.
\begin{enumerate}
\item If $0<q<p$ and $g\in \Q_{q}$ then $J_g: \Di_p^{q/p}\to \Di_p^{q/p} $ is bounded.\\
\item If $0<\la<1$ and $g\in \W_p$ then   $J_g:{\Di_p^{\la}}\to{\Di_p^{\la}}$ is bounded.
\end{enumerate}
\end{theorem}

\begin{proof}
(1) Set $\la=q/p<1$ and suppose $I\subset\T$ is an interval.  Using the  growth condition (\ref{growth})  for  $f\in {\Di_p^{\la}}$ we have
\begin{align*}
&\frac{1}{|I|^{p\la}}\int_{S(I)} |J_g(f)'(z)|^2(1-|z|^2)^ p\,dm(z) \\
=& \frac{1}{|I|^{q}}\int_{S(I)} |f(z)|^2|g'(z)|^2(1-|z|^2)^ p\,dm(z)\\
\lesssim &\frac{1}{|I|^{q}}\int_{S(I)} \frac{1}{(1-|z|^2)^{ p(1-\la)}}|g'(z)|^2(1-|z|^2)^ p\,dm(z)
\,\|f\|^2_{{\Di_p^{\la}}}\\
= &\frac{1}{|I|^{q}}\int_{S(I)}
|g'(z)|^2(1-|z|^2)^{q}\,dm(z)\,\|f\|^2_{{\Di_p^{\la}}}\\
 \lesssim   &\|g\|^2_{\Q_q}\|f\|^2_{{\Di_p^{\la}}},
\end{align*}
and the assertion follows by taking supremum on the left.

(2)  Let  $f\in{\Di_p^{\la}}$. For an interval $I\subset\T$  let  $w=w_I=(1-|I|)\eiteta$
where $\eiteta$ is the center of $I$.  Then
\begin{align*}
 &\frac{1}{|I|^{ p\la}}\int_{S(I)} |J_g(f)^\prime(z)|^2(1-|z|^2)^ p\,dm(z) \\
=&\frac{1}{|I|^{ p\la}}\int_{S(I)} |f(z)|^2|g^\prime(z)|^2(1-|z|^2)^ p\,dm(z)\\
 \leq  &\frac{2}{|I|^{ p\la}}\int_{S(I)} |f(w)|^2|g^\prime(z)|^2(1-|z|^2)^ p\,dm(z)\\
 & + \frac{2}{|I|^{ p\la}}\int_{S(I)} |f(z)-f(w)|^2|g^\prime(z)|^2(1-|z|^2)^ p\,dm(z)\\
=& A_I+B_I.
\end{align*}
For the first integral, using (\ref{growth}) and recalling that $\W_p\subset\Q_p$ we have
\begin{align*}
 A_I  \lesssim \|f\|^2_{\Di_p^{\la}} \frac{1}{|I|^{p}}\int_{S(I)} |g'(z)|^2(1-|z|^2)^ p\,dm(z)\lesssim
  \|f\|^2_{\Di_p^{\la}}\|g\|^2_{\Q_p}.
\end{align*}
 For the second integral we write
\begin{align*}
B_I&=\frac{2}{|I|^{ p\la}}\int_{S(I)} \left|\frac{f(z)-f(w)}{(1-\overline{w}z)^ p} \right|^2
|1-\overline{w}z|^{2 p}|g'(z)|^2(1-|z|^2)^ p\,dm(z)\\
&\lesssim |I|^{ p(2-\la)} \int_{S(I)} \left|\frac{f(z)-f(w)}{(1-\overline{w}z)^ p} \right|^2 |g'(z)|^2(1-|z|^2)^ p\,dm(z)\\
&= (1-|w|)^{ p(2-\la)}\int_{S(I)} \left|\frac{f(z)-f(w)}{(1-\overline{w}z)^ p} \right|^2 |g'(z)|^2(1-|z|^2)^ p\,dm(z).\\
&\lesssim  (1-|w|)^{ p(2-\la)} |f(0)-f(w)|^2 \\
&\quad +(1-|w|)^{ p(2-\la)} \int_{\D} \left|\frac{d}{dz}\left(\frac{f(z)-f(w)}{(1-\overline{w}z)^ p}\right) \right|^2 (1-|z|^2)^ p\,dm(z)\\
&=(1-|w|)^{ p(2-\la)} |f(0)-f(w)|^2 +C_w,
\end{align*}
where we have used  the hypothesis that $d\mu_g(z)=|g'(z)|^2(1-|z|^2)^ p\,dm(z)$ is a $\Di_p$-Carleson measure.
The first term in the last sum is
$$
(1-|w|)^{ p(2-\la)} |f(0)-f(w)|^2<(1-|w|)^{ p(1-\la)} |f(0)-f(w)|^2\lesssim\|f\|^2_{\Di_p^{\la}}
$$
by using (\ref{growth}) once more. For the second term  we have
 \begin{align*}
C_w = &(1-|w|)^{ p(2-\la)} \int_{\D} \left|\frac{d}{dz}\left(\frac{f(z)-f(w)}{(1-\overline{w}z)^ p}\right) \right|^2 (1-|z|^2)^ p\,dm(z)\\
 &= (1-|w|)^{ p(2-\la)} \int_{\D} |f^\prime(z)|^2\frac{(1-|z|^2)^ p}{|1-\overline{w}z|^{2 p} }\,dm(z)\\
 & \quad +|w|^2 (1-|w|)^{ p(2-\la)}\int_{\D} \left|\frac{f(z)-f(w)}{(1-\overline{w}z)^{1+ p} }\right|^2(1-|z|^2)^ p\,dm(z)\\
 & \sim (1-|w|)^{ p(1-\la)} \int_{\D} |f'(z)|^2 (1-|\varphi_{w}( z)|^2)^ p\,dm(z)\\
& \quad + |w|^2(1-|w|)^{ p(1-\la)} \int_{\D} \left|\frac{f(z)-f(w)}{1-\overline{w}z} \right|^2 (1-|\varphi_{w}( z)|^2)^ p\,dm(z)\\
&\lesssim \|f\|_{\Di_p^{\la}}^2 + (1-|w|)^{ p(1-\la)} \int_{\D}
\left|\frac{f(z)-f(w)}{1-\overline{w}z }\right|^2 (1-|\varphi_{w}( z)|^2)^ p\,dm(z)\\
&= \|f\|_{\Di_p^{\la}}^2+ D_w.
\end{align*}
Observe that
\begin{align*}
D_w&=(1-|w|)^{ p(1-\la)} \int_{\D} \left|\frac{f(z)-f(w)}{1-\overline{w}z }\right|^2 (1-|\varphi_{w}( z)|^2)^ p\,dm(z)\\
&=(1-|w|)^{ p(1-\la)} \int_{\D} \left|\frac{f \circ \varphi_{w}( z)-f \circ \varphi_{w}( 0)}{1-\overline{w}\varphi_{w}( z) }\right|^2
|\varphi_{w}'( z)|^2 (1-|z|^2)^ p\,dm(z)\\
&= (1-|w|)^{ p(1-\la)} \int_{\D} \left|\frac{f \circ \varphi_{w}( z)-f \circ \varphi_{w}( 0)}{1-\overline{w} z}\right|^2
 (1-|z|^2)^ p\,dm(z)
 \end{align*}
To find an upper estimate for $D_w$, we follow the argument of \cite{PP}, pages 551-552. See also \cite[page 2080]{Xi2}.
The argument consists of applying a reproducing formula from \cite{RW}, the
Cauchy-Schwarz inequality, Fubini's theorem and the estimate  \cite[Lemma 3.10(b)]{Zhu}. We refrain  from writing all the
details since the argument applies \textit{mutatis mutandis}. The final steps of the calculation are  as follows
 \begin{align*}
D_w &\lesssim (1-|w|)^{ p(1-\la)} \int_{\D} \left|(f\circ\varphi_{w})'(z)\right|^2
\frac{(1-| z|^2)^{2+ p}}{|1-\overline{w}z|^2}\,dm(z)\\
 &\lesssim (1-|w|)^{ p(1-\la)} \int_{\D} \left|(f\circ\varphi_{w})'(z)\right|^2 (1-| z|^2)^{ p}\,dm(z)\\
&\lesssim (1-|w|)^{ p(1-\la)} \int_{\D} \left|f'(z)\right|^2 (1-|\varphi_{w}( z)|^2)^ p\,dm(z)\\
&\lesssim \|f\|_{\Di_p^{\la}}^2.
\end{align*}

Collecting all the above estimates gives $\n{J_g(f)}_{\Di_p^{\la}}\leq C\n{f}_{\Di_p^{\la}}$ which is the desired conclusion.
\end{proof}

\par The above theorems in combination with (\ref{MulVol}) give the following  corollary for multipliers  of $\Di_p^{\la}$.

 \begin{corollary}\label{MUL} Suppose $ 0< p, \la <1$ and $g\in\hol(\D).$ Then
 \begin{enumerate}
\item\label{Mg1} If $g\in \W_{p}\cap H^{\infty}$
then
$M_g:{\Di_p^{\la}}\to{\Di_p^{\la}}$ is bounded.
\item\label{Mg2} If
$ g\in\mathcal Q_{ p \la} \cap H^{\infty}$
then
$M_g:{\Di_p^{\la}}\to {\Di_p^{\la}}$ is bounded
\item\label{Mg3} If $M_g:{\Di_p^{\la}}\to{\Di_p^{\la}}$ is bounded
then $
g\in \Q_p\cap H^{\infty}$.
\end{enumerate}
\end{corollary}

 \textbf{Remark.}
Let $0<p<1$. We know  that $\W_p\subset \Q_p$, and this inclusion is strict \cite[Theorem 6.3.4]{XiG}.
At the same time for   $0<q<p$ we have $\Q_q\subset\Q_p$ with strict inclusion.
For each $q<p$ we give an example of a function $f$ such that $f\in \W_p$ but $f$ does not belong to $\Q_q$. Thus
$\W_p\nsubseteq \Q_q$ for any $q<p$.

Indeed with  $q, p$ as above consider the function
\begin{equation}\label{HadGap}
 f(z)= \sum_{k=1}^{\infty}a_k  z^{2^k}
 \end{equation}
 where $a_k=1/2^{k(1-q)/2}$. By a theorem of Yamashita \cite[Theorem 1(i)]{Yam} for such
 Hadamard gap series, and since
 $$
\limsup |a_k|2^{k\left(1-\frac{1+q}{2}\right)}=1<\infty,
 $$
 it follows that $f$ satisfies the growth condition
 $$
\sup_{z\in\D}|f'(z)| (1-|z|)^{\frac{1+ q}{2}}<\infty.
$$
Applying Proposition 4.2 of  \cite{ALXZ} (after adjusting the parameters involved to our notation) we find that
this function is a multiplier of $\Di_p$ because  $q<p$. Thus the bounded function $f$ belongs to $\W_p$.

On the other hand
$$
\sum_{k= 0}^{\infty} 2^{k(1- q)} \left(\sum_{2^k \leq n_j < 2^{k+1}} |a_j|^2\right) =\sum_{k=0}^{\infty}1=\infty,
$$
and therefore by  \cite[Theorem 1.2.1]{Xi} for such Hadamard gap series, $f\notin \Q_q$.

The complete description of the multiplier space $M(\Di_p^{\la})$ and of the symbols
 $g$ for which $J_g$ is bounded on $\Di_p^{\la}$,  seems to be a hard problem.

 We would like to thank the referee for reading the article and for encouraging us to include the paragraph about the  construction of general Morrey-type  spaces $M(X, w)$.

\end{document}